\documentclass[11pt]{article}
\usepackage{amsfonts}
\usepackage{latexsym,amssymb,amsmath,graphics,cite,arydshln}
\usepackage{tikz}
\usepackage{fancyhdr}
\usepackage{lipsum}
\usepackage{lineno}
\usepackage{color}

\usepackage{algorithm} %format of the algorithm
\usepackage{algorithmic} %format of the algorithm
\usepackage{multirow} %multirow for format of table
\usepackage{amsmath}
\usepackage{xcolor}

\begin{document}
\newcommand{\qed}{\hphantom{.}\hfill $\Box$\medbreak}
\newcommand{\proof}{\noindent{\bf Proof \ }}

\newtheorem{theorem}{Theorem}[section]
\newtheorem{lemma}[theorem]{Lemma}
\newtheorem{corollary}[theorem]{Corollary}
\newtheorem{remark}[theorem]{Remark}
\newtheorem{example}[theorem]{Example}
\newtheorem{definition}[theorem]{Definition}
\newtheorem{construction}[theorem]{Construction}
\newtheorem{fact}[theorem]{Fact}
\newtheorem{proposition}[theorem]{Proposition}
%%%%%%%%%%%%%%%%%%%%%%%%%%%%%%%%%%%%%%%%%%%%%%%%%%%%%%%%%%%%%%%%%%%%%%%%

\begin{center}
{\Large\bf Several classes of optimal Ferrers diagram rank-metric codes \footnote{Supported by NSFC under Grant $11431003$ and $11871095$.}
}

\vskip12pt

Shuangqing Liu, Yanxun Chang, Tao Feng\\[2ex] {\footnotesize Department of Mathematics, Beijing Jiaotong University, Beijing 100044, P. R. China}\\
{\footnotesize
16118420@bjtu.edu.cn, yxchang@bjtu.edu.cn, tfeng@bjtu.edu.cn}
\vskip12pt

\end{center}

\vskip12pt

\noindent {\bf Abstract:} Four constructions for Ferrers diagram rank-metric (FDRM) codes are presented. The first one makes use of a characterization on generator matrices of a class of systematic maximum rank distance codes. By introducing restricted Gabidulin codes,
the second construction is presented, which unifies many known constructions for FDRM codes. The third and fourth constructions are based on two different ways to represent elements of a finite field $\mathbb F_{q^m}$ (vector representation and matrix representation). Each of these constructions produces optimal codes with different diagrams and parameters. % for which no optimal construction was known before.
\vskip12pt



%%%%%%%%%%%%%%%%%%%%%%%%%%%%%%%%%%%%%%%%%%%%%%%%%%%%%%%%%%%%%%%%%%%%%%%%%

\section{Introduction}

Let $\mathbb F_q$ be the finite field of order $q$, and $\mathbb F_q^n$ be the set of all vectors of length $n$ over $\mathbb F_q$. $\mathbb F_q^n$ is a vector space with dimension $n$ over $\mathbb F_q$. Given a nonnegative integer $k\leq n$, the set of all $k$-dimensional subspaces of $\mathbb F_q^n$, denoted by ${\cal G}_q(n,k)$, forms the {\em Grassmannian space} of order $n$ and dimension $k$ over $\mathbb F_q$. A nonempty subset of ${\cal G}_q(n,k)$ is called a {\em constant dimension code}, denoted by an $(n,k)_q$-CDC. The {\em subspace distance}
$d_S(\mathcal U,\mathcal V)={\rm dim}~\mathcal U+{\rm dim}~\mathcal V-2{\rm dim}~(\mathcal U\cap \mathcal V)$
for all $\mathcal U,\mathcal V\in {\cal G}_q(n,k)$ is used as a distance metric on ${\cal G}_q(n,k)$.

Constant dimension codes, motivated by their extensive application to error correction in random network coding, have become one of central topics in algebraic coding theory during the last ten years (see \cite{es09,es13,ev11,gr,hkkw,gy,kku,se13,st15,se131,s,tmbr,tr} for example). This interest stems from the groundbreaking work of K\"{o}tter and Kschischang \cite{kk}.

Error correction codes with the rank metric have been receiving steady attention in the literature due to their applications in network coding \cite{kk,skk} and storage systems \cite{r}. Let $\mathbb F^{m \times n}_q$ denote the set of all $m \times n$ matrices over $\mathbb F_q$. For a matrix $\boldsymbol{A} \in \mathbb F^{m \times n}_q$, the rank of $\boldsymbol{A}$ is denoted by rank$(\boldsymbol{A})$. $\mathbb F^{m \times n}_q$ is an $\mathbb F_q$-vector space. The {\em rank distance} on $\mathbb F^{m \times n}_q$ is defined by
\begin{center}
  $d_R(\boldsymbol{A}, \boldsymbol{B})={\rm rank}(\boldsymbol{A}-\boldsymbol{B})~{\rm for}~\boldsymbol{A}, \boldsymbol{B} \in \mathbb F^{m \times n}_q$.
\end{center}
An $[m \times n, k, \delta]_q$ {\em rank-metric code} $\mathcal C$ is a $k$-dimensional $\mathbb F_q$-linear subspace of $\mathbb F^{m \times n}_q$ with {\em minimum rank distance} $$\delta=\underset{\boldsymbol{A,B} \in \mathcal C, \boldsymbol{A}\neq \boldsymbol{B}}{\min}\{d_R(\boldsymbol{A},\boldsymbol{B})\}.$$
Clearly $$\delta=\underset{\boldsymbol{A} \in \mathcal C, \boldsymbol{A}\neq \boldsymbol{0}}{\min}\{{\rm rank}(\boldsymbol{A})\}.$$
The Singleton-like upper bound for rank-metric codes implies that $$k \leq {\rm max}\{m,n\}({\rm min}\{m,n\}- \delta +1)$$
holds for any $[m \times n, k, \delta]_q$ code. When the equality holds, $\cal C$ is called a {\em linear maximum rank distance code} (or an MRD$[m \times n, \delta]_q$ code in short). Linear MRD codes exists for all feasible parameters (cf. \cite{d,g,r}).

Write the $k\times k$ identity matrix as $\boldsymbol{I}_{k}$. Silva, Kschischang and K\"{o}tter \cite{skk} pointed out that for any MRD$[k\times (n-k),\delta]_q$ code $\cal C$, the lifted maximum rank distance code $\{{\rm row\ space\ of\ }(\boldsymbol{I}_k \mid \boldsymbol{A}):\boldsymbol{A} \in \mathcal C\}$ produces an $(n,k)_q$-CDC with minimum (nonzero) subspace distance $2\delta$, which asymptotically attain the known upper bounds on the size of constant dimension codes, and can be decoded efficiently in the context of random linear network coding. A further question is how to provide constructions for constant dimension codes which are larger than lifted MRD codes. The first progress was made by Etzion and Silberstein \cite{es09}, who presented a simple but effective construction, named the multilevel construction. This construction introduces a new family of rank-metric codes having a given shape of their codewords, namely, Ferrers diagram rank-metric codes.

Given positive integers $m$ and $n$, an $m \times n$ {\em Ferrers diagram} $\mathcal F$ is an $m \times n$ array of dots and empty cells such that (1) all dots are shifted to the right of the diagram; (2) the number of dots in each row is less than or equal to the number of dots in the previous row; (3) the first row has $n$ dots and the rightmost column has $m$ dots.

% We often label the columns of a Ferrers diagram from 1.
A Ferrers diagram $\mathcal F$ is identified with the cardinalities of its columns. Given positive integers $m$, $n$ and $1\leq \gamma_0 \leq \gamma_1 \leq \cdots \leq \gamma_{n-1}\leq m$, there exists a unique Ferrers diagram $\mathcal F$ of size $m\times n$ such that the $(j+1)$-th column of $\mathcal F$ has cardinality $\gamma_j$ for any $0\leq j\leq n-1$. In this case we write $\mathcal F=[\gamma_0,\gamma_1,\ldots,\gamma_{n-1}]$.
%5In this paper when we we always use $\gamma_{j-1}$ to denote the number of dots in the $j$-th column of a given Ferrers diagram $\mathcal F$.

For a given $m \times n$ Ferrers diagram $\mathcal F$, an $[\mathcal F, k, \delta]_q$ {\em Ferrers diagram rank-metric} (FDRM) {\em code}, briefly an $[\mathcal F, k, \delta]_q$ code, is an $[m \times n, k, \delta]_q$ rank-metric code in which for each $m \times n$ matrix, all entries not in $\mathcal F$ are zero.
%The maximum dimension of an $[\mathcal F, k, \delta]_q$ code is denoted by $\dim(\mathcal F,\delta)$ and clearly, $k\leq \dim(\mathcal F,\delta)$.
If $\mathcal F$ is a {\em full} $m\times n$ diagram with $mn$ dots, then its corresponding FDRM codes are just classical rank-metric codes.

Etzion and Silberstein \cite{es09} established a Singleton-like upper bound on FDRM codes.

\begin{lemma} {\rm (Theorem 1 in  \cite{es09})} \label{lem:upper bound}
Let $\delta$ be a positive integer. Let $v_i$, $0\leq i\leq \delta-1$, be the number of dots in a Ferrers diagram $\cal F$ which are not contained in the first $i$ rows and the rightmost $\delta-1-i$ columns. Then for any $[\mathcal F, k, \delta]_q$ code, $k\leq \min_{i\in\{0,1,\ldots,\delta-1\}}v_i$.
\end{lemma}

An FDRM code attaining the upper bound in Lemma \ref{lem:upper bound} is called {\em optimal}. Clearly, an MRD$[m \times n, \delta]_q$ code attains this upper bound for a full $m\times n$ Ferrers diagram.
% All known FDRM codes so far over $\mathbb F_q$ with the largest possible dimension are optimal.
Much work has been done on constructing good or even optimal FDRM codes \cite{al,b,egrw,es09,lf,gr,st15,zg}. The reader is referred to Section 2.3 in \cite{lf} for a survey of known constructions for FDRM codes. Here we only quote three of them, which will be used later.

The following theorem was first given by Etzion and Silberstein \cite{es09}, and its proof is simplified in \cite{egrw}.
%We only quote partial results as follows, which are useful later.

\begin{theorem} {\rm (Theorem 3 in \cite{egrw})} \label{thm:shortening}
Assume $\mathcal F=[\gamma_0,\gamma_1,\ldots,\gamma_{n-1}]$ is an $m \times n$ Ferrers diagram and each of the rightmost $\delta-1$ columns of $\mathcal F$ has at least $n$ dots. Then there exists an optimal $[\mathcal F, \sum_{i=0}^{n-\delta} \gamma_i, \delta]_q$ code for any prime power $q$.
\end{theorem}

Construction 2 in \cite{egrw} presented a method to obtain optimal FDRM codes by exploring subcodes of MRD codes, where each of the rightmost $\delta-1$ columns in Ferrers diagram $\cal F$ is required to have at least $n-1$ dots.

\begin{theorem}{\rm (Theorem 8 in \cite{egrw})}\label{thm:subcodes from Gab}
Let $\delta$ and $n$ be positive integers satisfying $2\leq \delta \leq n-1$. Let $\mathcal{F}=[\gamma_0,\gamma_1,\ldots,\gamma_{n-1}]$ be an $m\times n$ Ferrers diagram satisfying that
\begin{itemize}
\item[$(1)$] $\gamma_{n-1}\geq n-1+\gamma_0$,
\item[$(2)$] $\gamma_{n-\delta+1}\geq n-1$.
\end{itemize}
Then there exists an optimal $[\mathcal F, \sum_{i=0}^{n-\delta} \gamma_i, \delta]_q$ code for any prime power $q$.
\end{theorem}

Etzion, Gorla, Ravagnani and Wachter-Zeh in \cite{egrw} also presented a way to obtain an FDRM code by combining two FDRM codes with the same dimension.

\begin{theorem}{\rm (Theorem 9 in  \cite{egrw})} \label{thm:combine with same dim}
Let $\mathcal F_i$ for $i=1,2$ be an $m_i \times n_i$ Ferrers diagram, and $\mathcal C_i$ be an $[\mathcal F_i, k, \delta_i]_q$ code. Let $\mathcal D$ be an $m_3 \times n_3$ full Ferrers diagram with $m_3n_3$ dots, where $m_3 \geq m_1$ and $n_3 \geq n_2$. Let
\begin{center}
$\mathcal F=\left(
  \begin{array}{cc}
    \mathcal F_1 & \mathcal D \\
      & \mathcal F_2 \\
  \end{array}
\right)$
\end{center}
be an $m \times n$ Ferrers diagram $\mathcal F$, where $m=m_2+m_3$ and $n=n_1+n_3$. Then there exists an $[\mathcal F, k, \delta_1+\delta_2]_q$ code.
\end{theorem}

\begin{example}\label{eg:2,3,3,5}
Let
\begin{center}
$\mathcal F=
    \begin{array}{cccc}
     \bullet &  \bullet & \bullet & \bullet  \\
      \bullet &  \bullet & \bullet & \bullet  \\
      & \bullet & \bullet & \bullet  \\
           &  &  & \bullet  \\
               &  &  & \bullet
    \end{array}$, \ \ \
$  \mathcal F_1=
    \begin{array}{ccc}
     \bullet &  \bullet & \bullet  \\
      \bullet &  \bullet & \bullet  \\
      & \bullet & \bullet
    \end{array}$\ \ \ and \ \ \
$\mathcal F_2=
    \begin{array}{c}
 \bullet  \\
 \bullet
\end{array}$.
\end{center}
By Theorem $\ref{thm:combine with same dim}$, there exists an optimal $[\mathcal F,2,4]_{q}$ code for any prime power $q$, where the needed $[\mathcal F_1,2,3]_{q}$ code and $[\mathcal F_2,2,1]_{q}$ code come from Theorem $\ref{thm:shortening}$. Its optimality is guaranteed by Lemma $\ref{lem:upper bound}$.
\end{example}

This paper continues the study in \cite{lf} to establish more constructions for optimal FDRM codes. Throughout this paper, let $[m]=\{0,1,\ldots,m-1\}$. For a matrix $\boldsymbol{A} \in \mathbb F^{m \times n}_q$, the rows and columns of $\boldsymbol{A}$ are indexed by $0,1,\ldots, m-1$ and $0,1,\ldots, n-1$, respectively. Denote by
$\boldsymbol{A}(i,j)$ the value in the $i$-th row and the $j$-th column of the matrix $\boldsymbol{A}$, where $i\in [m]$ and $j\in [n]$.

In Section 2.1, by using a description on generator matrices of a class of systematic MRD codes presented in \cite{al}, we give a class of optimal FDRM codes (see Theorem \ref{cfdrm}), which relaxes the condition on the number of dots in the $(\delta-1)$-th column from the right end compared with Theorem \ref{thm:shortening}. In Section 2.2, via restricted
Gabidulin codes, we obtain another class of optimal FDRM codes (see Theorem \ref{the:from sys MRD-new2}), which generalizes Theorem 3.11 in \cite{lf}, and Theorems 3.2 and 3.6 in \cite{zg}. Section 3 provides two constructions for FDRM codes based on two different ways to represent elements of a finite field $\mathbb F_{q^m}$ (vector representation and matrix representation). Concluding remarks are made in Section 4.

\section{Constructions based on subcodes of MRD codes}

Let $\mathbb F_q$ be the finite field of order $q$, and $\mathbb F_{q^m}$ be its extension field of order $q^m$. We use $\mathbb F^n_{q^m}$ to denote the set of all row vectors of length $n$ over $\mathbb F_{q^m}$. Let \boldmath $\mathbf{\beta}$ \unboldmath = $(\beta_0,  \beta_1, \dots, \beta_{m-1})$ be an ordered basis of $\mathbb F_{q^m}$ over $\mathbb F_q$. There is a natural bijective map $\Psi_m$ from $\mathbb F_{q^m}^n$
to $\mathbb F_{q}^{m \times n}$ as follows:
\renewcommand\theequation{\arabic{section}.\arabic{equation}}
 \begin{equation}\label{iso}
 \Psi_m: \mathbb F^n_{q^m} \longrightarrow \mathbb F^{m\times n}_q
\end{equation}
\begin{center}
  $\boldsymbol{a}=(a_0, a_1, \ldots, a_{n-1}) \longmapsto \boldsymbol{A},$
\end{center}
where $\boldsymbol{A}=\Psi_m(\boldsymbol{a})\in \mathbb F^{m\times n}_q$ is defined such that
\begin{equation*}
 a_j=\sum_{i=0}^{m-1} \boldsymbol{A}(i, j)\beta_i
\end{equation*}
for any $j\in \{0,1,\ldots,n-1\}$. For $a\in \mathbb F_{q^m}$, $(a)$ is a $1\times 1$ matrix and we simply write $\Psi_m((a))$ as $\Psi_m(a)$. It is readily checked that $\Psi_m$ satisfies linearity, i.e., $\Psi_m(x\boldsymbol{a}_1+y\boldsymbol{a}_2)=x\Psi_m(\boldsymbol{a}_1)+y\Psi_m(\boldsymbol{a}_2)$ for any $x,y\in \mathbb F_{q}$ and $\boldsymbol{a}_1,\boldsymbol{a}_2\in \mathbb F^n_{q^m}$.
The map $\Psi_m$ will be used to facilitate switching between a vector in $\mathbb F_{q^m}$ and its matrix representation over $\mathbb F_q$. In the sequel, we use both representations, depending on what is more convenient in the context and by slight abuse of notation, rank$(\boldsymbol{a})$ denotes rank$({\Psi_m}(\boldsymbol{a}))$.

\subsection{Construction from systematic MRD codes}

MRD codes play an important role in the constructions for FDRM codes. The following lemma was first implicitly shown in Section 5 in \cite{egrw}, and then given in Lemma 3.1 in \cite{lf}. However, Lemma 3.1 in \cite{lf} does not emphasize the optimality of the resulting FDRM code and more importantly, does not point out the number of dots in each of the leftmost $k$ columns of its corresponding Ferrers diagram explicitly, which will be useful later (see Theorems \ref{cfdrm}). Thus we restate this lemma here and provide more details on its proof.

\begin{lemma}{\rm \cite{egrw,lf}}\label{lem:subcode from MRD}
Assume that $m\geq n$. Let $\boldsymbol{G}$ be a generator matrix of a systematic MRD$[m\times n,\delta]_q$ code, i.e., $\textbf{G}$ is of the form $(\boldsymbol{I}_k|\boldsymbol{A})$, where $k=n-\delta+1$. Let $0\leq \lambda_0 \leq \lambda_1\leq \cdots \leq \lambda_{k-1}\leq m$. Let ${U}=\{(u_{0},\ldots,u_{k-1}) \in \mathbb F_{q^{m}}^k:\Psi_{m}(u_i)=(u_{i,0},\ldots,u_{i,\lambda_{i}-1},0,\ldots,0)^T, u_{i,j}\in \mathbb F_q,i\in [k], j\in [\lambda_i]\}$. Then
$$\mathcal C=\{\Psi_m(\boldsymbol{c}):\boldsymbol{c}=\boldsymbol{u}\boldsymbol{G},\boldsymbol{u}\in {U}\}$$ is an optimal $[\mathcal F,\sum_{i=0}^{k-1} \lambda_i,\delta]_q$ code, where $\mathcal F=[\gamma_0,\gamma_1,\ldots,\gamma_{n-1}]$ satisfies $\gamma_i=\lambda_i$ for each $i\in [k]$ and $\gamma_i=m$ for $k\leq i\leq n-1$.
\end{lemma}

\proof The linearity and the dimension of $\mathcal C$ can be verified straightforwardly.  Since $\textbf{G}$ is a generator matrix of an MRD$[m \times n, \delta]_{q}$ code $\mathcal C_{M}$, $\mathcal C$ is a subcode of $\mathcal C_{M}$. It follows that the minimum rank distance of $\mathcal C$ is $\delta$.

Each codeword in $\mathcal C$ is of the form $\Psi_m(\boldsymbol{c})$, where $\boldsymbol{c}=(e_0,e_1,\ldots,e_{n-1})=\boldsymbol{uG}$ for some $\boldsymbol{u}=(u_0,u_1,\ldots,u_{k-1})\in \mathbb F_{q^m}^k$.
For $i\in [k]$, $e_i=u_i,$ and so $\Psi_m(e_i)=\Psi_m(u_i)=(u_{i,0},\ldots,u_{i,\lambda_i-1},0,\ldots,0)$. Therefore, $\mathcal C$ is an $[\mathcal F,\sum_{i=0}^{k-1} \lambda_i,\delta]_q$ code, where $\mathcal F=[\gamma_0,\gamma_1,\ldots,\gamma_{n-1}]$ satisfies $\gamma_i=\lambda_i$ for each $i\in [k]$. Its optimality follows from Lemma \ref{lem:upper bound} by examining the value of $v_0$. \qed

Lemma \ref{lem:subcode from MRD} only gives details of the leftmost $k$ columns of the Ferrers diagram used for codewords in $\cal C$. If we could know more about the initial MRD code, then it would be possible to give a complete characterization of $\cal C$. In Lemma 3.4 in \cite{lf}, we presented a class of systematic MRD codes and applied them to construct optimal FDRM codes. Here we shall make use of another class of systematic MRD codes from \cite{al} to produce more optimal FDRM codes.

\begin{lemma} {\rm (Lemma $3.13$ in \cite{al})} \label{dmrd}
Let $m\geq n\geq \delta\geq 2$ and $k=n-\delta+1$. For any prime power $q$ and any $a_1,a_2,\ldots,a_k \in \mathbb F_{q^m}$ satisfying that $1,a_1,a_2,\ldots,a_k$ are linearly independent over $\mathbb F_{q}$, there exists a matrix $\boldsymbol{A} \in \mathbb F^{k \times (n-k)}_{q^m}$ such that its first column is given by $(a_1,\ldots,a_k)^T$ and $\textbf{G}=(\boldsymbol{I}_k|\boldsymbol{A})$ is a generator matrix of a systematic MRD$[m\times n,\delta]_q$ code.
\end{lemma}

For a vector $(v_1,v_2,\ldots,v_n)$ of length $n$, if its rightmost nonzero component is $v_r$ for some $1\leq r\leq n$, then $r$ is said to be the {\em valid length} of this vector.

\begin{theorem}  \label{cfdrm}
Let $m\geq n\geq \delta\geq 2$ and $k=n-\delta+1$. If an $m\times n$ Ferrers diagram $\mathcal F=[\gamma_0,\gamma_1,\ldots,\gamma_{n-1}]$ satisfies
\begin{itemize}
\item[$(1)$] $\gamma_{k}\geq n$ or $\gamma_{k}-k\geq \gamma_i-i$ for each $i=0,1,\ldots,k-1$,
\item[$(2)$] $\gamma_{k+1}\geq n$,
\end{itemize}
then there exists an optimal $[\mathcal F,\sum_{i=0}^{k-1} \gamma_i,\delta]_q$ code for any prime power $q$.
\end{theorem}

\begin{proof} If $\gamma_{k}\geq n$, then since $\mathcal F$ is a Ferrers diagram, we have $\gamma_i\geq n$ for any $k\leq i\leq n-1$. Thus each of the rightmost $\delta-1$ columns of $\mathcal F$ has at least $n$ dots. By Theorem \ref{thm:shortening}, there exists an optimal $[\mathcal F, \sum_{i=0}^{k-1} \gamma_i, \delta]_q$ code for any prime power $q$.

If $\gamma_{k}<n$, then let $(1,\beta,\beta^2,\ldots,\beta^{n-1})$ be an ordered basis of $\mathbb F_{q^n}$ over $\mathbb F_q$. Note that $n\geq \delta\geq 2$, so $n-1\geq k$. We can apply Lemma \ref{dmrd} with $a_i=\beta^{i}$ for $1\leq i\leq k$ to obtain a matrix $\boldsymbol{A} \in \mathbb F^{k \times (n-k)}_{q^n}$ such that its first column is given by $(\beta^k,\beta^{k-1},\ldots,\beta)^T$ and $\textbf{G}=(\boldsymbol{I}_k|\boldsymbol{A})$ is a generator matrix of a systematic MRD$[n\times n,\delta]_q$ code. Then apply Lemma \ref{lem:subcode from MRD} by setting $\lambda_i=\gamma_i$ for $i\in [k]$ to obtain an optimal $[\mathcal F',\sum_{i=0}^{k-1} \gamma_i,\delta]_q$ code $\mathcal C$, where $\mathcal F'=[\gamma'_0,\gamma'_1,\ldots,\gamma'_{n-1}]$ satisfies $\gamma'_i=\gamma_i$ for each $i\in [k]$. We shall analyze the number of dots in each column of $\mathcal F'$ and show that $\mathcal F'\subseteq \mathcal F$, i.e., $\gamma'_i\leq \gamma_i$ for each $i\in [n]$, which implies the existence of an optimal $[\mathcal F, \sum_{i=0}^{k-1} \gamma_i, \delta]_q$ code by Lemma \ref{lem:upper bound}.

By Lemma \ref{lem:subcode from MRD}, each codeword in $\mathcal C$ is of the form $\Psi_n(\boldsymbol{c})$, where $\boldsymbol{c}=(e_0,e_1,\ldots,e_{n-1})=\boldsymbol{uG}$ for some $\boldsymbol{u}=(u_0,u_1,\ldots,u_{k-1})\in \mathbb F_{q^n}^k$.

For $i=k$, $e_{k}=\sum_{j=0}^{k-1} u_j \beta^{k-j}$ and so $\Psi_n(e_{k})=\sum_{j=0}^{k-1} \Psi_n(u_j\beta^{k-j})$. For $0\leq j\leq k-1$, $\Psi_{n}(u_j)=(u_{j,0},u_{j,1},\ldots,u_{j,\gamma_{j}-1},0,\ldots,0)^T$ implies  $u_j=u_{j,0}+u_{j,1}\beta+\cdots+u_{j,\gamma_j-1}\beta^{\gamma_j-1}$. It follows that for each $j\in [k]$, as a vector of length $n$, $\Psi_n(u_j\beta^{k-j})$ has a valid length of at most $\min\{\gamma_j+k-j,n\}$. Thus $\Psi_n(e_{k})$ has a valid length of at most $\max_{j\in [k]}\{\gamma_j+k-j\}$ if $\max_{j\in [k]}\{\gamma_j+k-j\}\leq n$, or $n$ otherwise. By Condition $(1)$, $\gamma'_k\leq \gamma_k$.

For $k+1\leq i\leq n-1$, $\Psi_n(e_i)$ has a valid length of at most $n$. By Condition $(2)$, $\gamma'_i\leq \gamma_i$. \qed
\end{proof}

Compared with Theorem $\ref{thm:shortening}$ in which each of the rightmost $\delta-1$ columns of $\mathcal F$ consists of at least $n$ dots, Theorem $\ref{cfdrm}$ requires each of the rightmost $\delta-2$ columns of $\mathcal F$ has at least $n$ dots and relaxes the condition on $\gamma_{n-\delta+1}$.
%the number $\gamma_{n-\delta+1}$ of dots in the $(\delta-1)$-th column from the right end.

%\begin{example} \label{fdrm4}
%Let $l$ be a nonnegative integer and $\mathcal F=[l+1,l+2,\ldots,l+n-2,l+n,l+n]$ be an $(l+n)\times n$ Ferrers diagram. By Theorem $\ref{cfdrm}$, there exists an optimal $[\mathcal F,\frac{(n-3)(n-2)}{2}+l(n-3),4]_q$ code for any integer $n\geq 4$ and any prime power $q$.
%\end{example}

%Applying Theorem $\ref{thm:from MDS}$, one can obtain an optimal $[\mathcal F,\frac{(n-3)(n-2)}{2}+l(n-3),4]_q$ code for any prime power $q\geq n-1$. However, Theorem $\ref{thm:from MDS}$ does not work for $q<n-1$.

\begin{example} \label{fdrm5}
Let $n\geq \delta\geq 2$, $k=n-\delta+1$ and $\mathcal F=[\gamma_0,\gamma_1,\ldots,\gamma_{n-1}]$ be a Ferrers diagram satisfying $\gamma_{i+1}\geq \gamma_i+1$ for $i\in [k]$ and $\gamma_{i}\geq n$ for $i\in [n]\setminus[k+1]$. By Theorem $\ref{cfdrm}$, there exists an optimal $[\mathcal F,\sum_{i=0}^{k-1} \gamma_i,\delta]_q$ code for any prime power $q$.
\end{example}

\begin{example} \label{fdrm55}
Let $\mathcal F=[2,2,\gamma_2,\ldots,\gamma_{n-4},n-1,n,n]$ be an $n\times n$ Ferrers diagram, where $\gamma_i\leq i+2$ for $2\leq i\leq n-4$. By Theorem $\ref{cfdrm}$, there exists an optimal $[\mathcal F,\sum_{i=2}^{n-4}\gamma_i+4,4]_q$ code for any integer $n\geq 6$ and any prime power $q$. When $n=4$, for $\mathcal F=[2,3,4,4]$, by Theorem $\ref{cfdrm}$, there exists an optimal $[\mathcal F,2,4]_q$ code for any prime power $q$. When $n=5$, for $\mathcal F=[2,2,4,5,5]$, by Theorem $\ref{cfdrm}$, there exists an optimal $[\mathcal F,4,4]_q$ code for any prime power $q$.
\end{example}

We remark that by using Corollary 3.11 in \cite{al}, one can also obtain Example \ref{fdrm55}, but cannot obtain Example \ref{fdrm5}. For example one can compare Corollary 3.11 in \cite{al} and Example \ref{fdrm5} by examining the case of $\mathcal F=[1,2,\ldots,n-2,n,n]$ (an $n\times n$ Ferrers diagram) and $n\geq \delta=4$.

%It seems that constructing an optimal $[\mathcal F,4]_q$ code for any $n\times n$ Ferrers diagram $\mathcal F=[2,2,\gamma_2,\ldots,\gamma_{n-4},n-1,n,n]$ is still an open problem (see details in Section 5).

\subsection{Constructions based on subcodes of restricted Gabidulin codes}

Gabidulin codes are a special class of MRD codes. Let $m\geq n$ and $q$ be any prime power. Let $\delta$ be a positive integer. For any positive integer $i$ and any $a\in \mathbb F_{q^m}$, set $a^{[i]} \triangleq a^{q^i}$. A {\em Gabidulin code} $\mathcal{G}[m\times n,\delta]_q$ is an MRD$[m\times n,\delta]_q$ code whose generator matrix $\boldsymbol{G}$ in vector representation is
\renewcommand\theequation{\arabic{section}.\arabic{equation}}
\begin{equation}\label{gm}
\boldsymbol{G}=\left(
               \begin{array}{cccc}
                 g_0 & {g_1} & \cdots & g_{n-1} \\
                 g_0^{[1]} & g_1^{[1]} & \cdots & g_{n-1}^{[1]} \\
                 \vdots & \vdots & \ddots & \vdots \\
                 g_0^{[n-\delta]} & g_1^{[n-\delta]} & \cdots & g_{n-1}^{[n-\delta]} \\
               \end{array}
             \right)
\end{equation}
where $g_0, g_1, \ldots, g_{n-1} \in \mathbb F_{q^m}$ are linearly independent over $\mathbb F_q$ (see \cite{g}).

Let $l\geq 1$ and $1=t_0<t_1<t_2<\cdots<t_l$ be integers such that $t_1\mid t_2\mid \cdots\mid t_l$ and $t_{l-1}< n\leq t_l$. Let $t_{x}=s_{x}t_{x-1}$ for $1\leq x\leq l$. Since $t_1\mid t_2\mid \cdots\mid t_l$, we have $\mathbb F_{q^{t_1}}\subset\mathbb F_{q^{t_2}}\subset\cdots\subset\mathbb F_{q^{t_l}}$.
For $1\leq x\leq l$, let $$(\alpha_{x,0}=1, \alpha_{x,1}, \ldots, \alpha_{x,s_{x}-1})$$
be an ordered basis of $\mathbb F_{q^{t_{x}}}$ over $\mathbb F_{q^{t_{x-1}}}$. Note that $s_1=t_1$. Let \begin{equation}\label{basis-1}
\beta_z=\alpha_{1,z-1}\in \mathbb F_{q^{t_1}}
\end{equation}
for $1\leq z\leq t_1$ and
\begin{equation}\label{basis-2}
\beta_{yt_{x-1}+z}=\beta_{z}\alpha_{x,y}\in \mathbb F_{q^{t_x}}
\end{equation}
for $2\leq x\leq l$, $1\leq y\leq s_{x}-1$ and $1\leq z\leq t_{x-1}$.  Then
\begin{equation*}\label{basis}
(\beta_1=1,\beta_2,\ldots,\beta_{t_l})
\end{equation*}
is said to be {\em an ordered basis of $\mathbb F_{q^{t_l}}$ over $\mathbb F_q$ with respect to $(t_0,t_1,\ldots,t_l)$}. Let $\boldsymbol G'$ be a generator matrix in the form of (\ref{gm}) of a Gabidulin code $\mathcal G[t_l\times n,\delta]_{q}$, where $g_{j-1}=\beta_j$ for $1\leq j\leq n$ (note that $n\leq {t_l}$). Then we refer to such a Gabidulin code as a {\em restricted Gabidulin code with respect to $(\beta_1,\beta_2,\ldots,\beta_{t_l})$}.
%Similarly,  Let $\boldsymbol G'$ be a matrix in the form of (\ref{tgm}) of a twisted Gabidulin code $\mathcal G[t_l\times n,\delta]_{q}$, where $g_{j-1}=\beta_j$ for $1\leq j\leq n$ (note that $n\leq {t_l}$). Then we refer to such a twisted Gabidulin code as a {\em restricted twisted Gabidulin code with respect to $(\beta_1,\beta_2,\ldots,\beta_{t_l})$}.

\begin{proposition}\label{prop:basis}
Let $l\geq 1$ and $1=t_0<t_1<t_2<\cdots<t_l$ be integers such that $t_1\mid t_2\mid \cdots\mid t_l$. Let $(\beta_1=1,\beta_2,\ldots,\beta_{t_l})$ be an ordered basis of $\mathbb F_{q^{t_l}}$ over $\mathbb F_q$ with respect to $(t_0,t_1,\ldots,t_l)$. Let $t_{2}=s_{2}t_{1}$ if $l>1$. Take $w\in\{1,2,\ldots,s_2\}$ or $w=1$ if $l=1$. Then
\begin{itemize}
\item[$(1)$] $\beta_{i}\beta_j\in \{\beta_1, \beta_2, \ldots, \beta_{wt_1}\}$ for any $1\leq i\leq t_1$ and $1\leq j\leq wt_1$;
\item[$(2)$] $\beta_{i}\beta_j\in \mathbb F_{q^{t_{\theta+1}}}$ for any $1\leq \theta \leq l-1$, $1\leq i\leq t_{\theta+1}$ and $1\leq j\leq wt_1$.
\end{itemize}
\end{proposition}

\proof By \eqref{basis-1} and \eqref{basis-2}, we have $\{\beta_1, \beta_2, \ldots, \beta_{wt_1}\}=
\{\beta_1, \beta_2, \ldots, \beta_{t_1},
\beta_{1}\alpha_{2,1},\beta_2\alpha_{2,1}$, $\ldots, \beta_{t_1}\alpha_{2,1},\ldots,\beta_1\alpha_{2,w-1},\beta_2\alpha_{2,w-1},\ldots,\beta_{t_1}\alpha_{2,w-1}\}$. Since $(\beta_1, \beta_2, \ldots, \beta_{t_1})$ is an ordered basis of $\mathbb F_{q^{t_{1}}}$ over $\mathbb F_{q}$, $(1)$ follows immediately. Due to $\mathbb F_{q^{t_1}}\subset\mathbb F_{q^{t_2}}\subset\cdots\subset\mathbb F_{q^{t_l}}$, $(2)$ holds. \qed

%\subsubsection{Construction based on subcodes of restricted Gabidulin codes}
The following lemma is a generalization of Lemma $3.11$ in \cite{lf}, which only deals with the case of $l=1$. Lemma $3.11$ in \cite{lf} is a generalization of Lemma 5 in \cite{egrw}.

\begin{lemma}\label{lem:subcode from Gab}
Let $l$ be a positive integer and $1=t_0<t_1<t_2<\cdots<t_l$ be integers such that $t_1\mid t_2\mid \cdots\mid t_l$. Let $r$ be a nonnegative integer, and $\eta,d,\kappa$ be positive integers satisfying $t_{l-1}<\eta-r\leq t_l$, $\kappa=\eta-r-d+1$ and $r<\kappa\leq t_1$. Then there exists a matrix $ \textbf{G}=(\boldsymbol I_{\kappa}|\boldsymbol A_1|\cdots|\boldsymbol A_{l})\in \mathbb F_{q^{t_l}}^{\kappa \times \eta}$ of the following form
\begin{center}\begin{scriptsize}
$\left(
                           \begin{array}{cccccccccccccc}
                              1 &   &   &   &   &   &   & a_{0,\kappa} & \cdots & a_{0,\eta-r-1} & 0 & 0 & \cdots & 0 \\
                                & 1  &   &   &   &   &   & a_{1,\kappa} & \cdots & a_{1,\eta-r-1} & a_{1,\eta-r} & 0 & \cdots & 0 \\
                                &   & \ddots  &   &  &    &   & \vdots & \ddots & \vdots & \vdots & \vdots & \ddots & \vdots \\
                                &   &   & 1  &   &   &   & a_{r-1,\kappa} & \cdots & a_{r-1,\eta-r-1} & a_{r-1,\eta-r} & a_{r-1,\eta-r+1} & \cdots & 0\\
                                &   &   &   & 1  &   &   & a_{r,\kappa} & \cdots & a_{r,\eta-r-1} & a_{r,\eta-r} & a_{r,\eta-r+1} & \cdots & a_{r,\eta-1}\\
                                &   &   &   &   & \ddots  &   & \vdots & \ddots & \vdots & \vdots & \vdots & \ddots & \vdots\\
                                &   &   &   &   &   & 1 & a_{\kappa-1,\kappa} & \cdots & a_{\kappa-1,\eta-r-1}  & a_{\kappa-1,\eta-r} & a_{\kappa-1,\eta-r+1} & \cdots & a_{\kappa-1,\eta-1} \\
                           \end{array}
                         \right)$,\end{scriptsize}
\end{center}
where
$$\boldsymbol A_1\in \left\{
\begin{array}{lll}
\mathbb F^{\kappa\times (t_1-\kappa)}_{q^{t_1}}, &  {\rm if}\ l\geq 2, \\
\mathbb F^{\kappa\times (\eta-\kappa)}_{q^{t_1}}, & {\rm if}\ l=1,
\end{array}
\right.
$$
$\boldsymbol A_x\in \mathbb F^{\kappa\times {(t_x-t_{x-1})}}_{q^{t_x}}$ for $2\leq x\leq l-1$, and $\boldsymbol A_l\in \mathbb F^{\kappa\times (\eta-t_{l-1})}_{q^{t_l}}$ if $l\geq 2$,
satisfying that for each $0\leq \nu\leq r$, the sub-matrix obtained by removing the first $\nu$ rows, the leftmost $\nu$ columns and the rightmost $r-\nu$ columns of $\boldsymbol{G}$ produces a systematic MRD$[t_l \times (\eta-r),d+\nu]_q$ code.
\end{lemma}

%We remark that $\textbf{G}$ is a $\kappa\times \eta$ matrix, so when $r=0$, $\textbf{G}$ is of the following form
%\begin{center}
%$\left(
%\begin{array}{cccccccccccccc}
%                               1  &   &   &   &   \alpha_{0,\kappa} & \cdots & \alpha_{0,\eta-1}  \\
%                                  &  1 &   &   &   \alpha_{1,\kappa} & \cdots & \alpha_{1,\eta-1} \\
%                                 & & \ddots  &  &    \vdots & \ddots & \vdots  \\                             &   &   & 1  &   \alpha_{\kappa-1,\kappa} & \cdots & \alpha_{\kappa-1,\eta-1}  \\
%                           \end{array}
%                         \right)$.
%\end{center}

%\vspace{0.4cm}

\begin{proof} (Sketch only) Let $(\beta_1=1,\beta_2,\ldots,\beta_{t_l})$ be an ordered basis of $\mathbb F_{q^{t_l}}$ over $\mathbb F_q$ with respect to $(t_0,t_1,\ldots,t_l)$.
Since $t_{l-1}<\eta-r\leq t_l$, we can take a restricted Gabidulin code $\mathcal{G}[t_l\times (\eta-r), d]_q$ with respect to $(\beta_1,\beta_2,\ldots,\beta_{t_l})$, whose generator matrix in vector representation is
\begin{center}
$\boldsymbol {G_0}=\left(
                   \begin{array}{cccc}
                     1 & g_{0,1} & \cdots & g_{0,\eta-r-1} \\
                     1 & g_{0,1}^{[1]} & \cdots & g_{0,\eta-r-1}^{[1]} \\
                     \vdots & \vdots & \ddots & \vdots \\
                     1 & g^{[\kappa-1]}_{0,1} & \cdots & g_{0,\eta-r-1}^{[\kappa-1]} \\
                   \end{array}
                 \right),$
\end{center}
where $g_{0,j}=\beta_{j+1}$ for $1\leq j\leq \eta-r-1$.

Using suitable row transformations, $\textbf G_0$ can be extended to produce $\textbf G_r$ by adding $r$ new columns. We need $r$ steps. For $0\leq i\leq r-1$, in Step $i$, let $\omega_i=\eta-r+i-2$ and $\textbf{G}_i=$
\begin{center}\setlength{\arraycolsep}{3.5pt}
%\begin{scriptsize}
\begin{tiny}
$\left(
\begin{array}{c;{2pt/2pt}ccccccccccccc}
      &      & 0 & 0 & \cdots & 0 & a_{0,\kappa} & \cdots & a_{0,\eta-r-1} & 0 & 0 & \cdots & 0 & 0 \\
      &      & 0 & 0 & \cdots & 0 & a_{1,\kappa} & \cdots & a_{1,\eta-r-1} & a_{1,\eta-r} & 0 & \cdots & 0 & 0 \\
    \textbf{I}_{i}    &    & \vdots & \vdots & \ddots & \vdots & \vdots & \ddots& \vdots & \vdots & \vdots & \ddots &\vdots & \vdots\\
      &      & 0 & 0 & \cdots & 0 & a_{i-2,\kappa} & \cdots & a_{i-2,\eta-r-1} & a_{i-2,\eta-r} & a_{i-2,\eta-r+1} & \cdots & 0 & 0 \\
      &      & 0 & 0 & \cdots & 0 & a_{i-1,\kappa} & \cdots & a_{i-1,\eta-r-1} & a_{i-1,\eta-r} & a_{i-1,\eta-r+1} & \cdots & a_{i-1,\omega_i} & 0 \\ \hdashline[2pt/2pt]
      &      & 1 & g_{i,i+1} & \cdots & g_{i,\kappa-1} & g_{i,\kappa} & \cdots & g_{i,\eta-r-1} & g_{i,\eta-r} & g_{i,\eta-r+1} & \cdots & g_{i,\omega_i} & g_{i,\omega_i+1} \\
      &      & 1 & g^{[1]}_{i,i+1} & \cdots & g^{[1]}_{i,\kappa-1} & g^{[1]}_{i,\kappa} & \cdots & g^{[1]}_{i,\eta-r-1} & g^{[1]}_{i,\eta-r} & g^{[1]}_{i,\eta-r+1} & \cdots & g^{[1]}_{i,\omega_i} & g^{[1]}_{i,\omega_i+1}\\
      &      & \vdots &\vdots & \ddots & \vdots & \vdots & \ddots & \vdots & \vdots & \vdots & \ddots & \vdots & \vdots \\
      &      & 1 & g^{[\kappa-i-1]}_{i,i+1} & \cdots & g^{[\kappa-i-1]}_{i,\kappa-1} & g^{[\kappa-i-1]}_{i,\kappa} & \cdots & g^{[\kappa-i-1]}_{i,\eta-r-1} & g^{[\kappa-i-1]}_{i,\eta-r} & g^{[\kappa-i-1]}_{i,\eta-r+1}  & \cdots & g^{[\kappa-i-1]}_{i,\omega_i} & g^{[\kappa-i-1]}_{i,\omega_i+1} \\
   \end{array}
 \right)$\end{tiny}%\end{scriptsize}
\end{center}
be a $\kappa\times (\omega_i+2)$ matrix, where
\begin{itemize}
\item[$(1)$] $1, g_{i,i+1},\ldots,g_{i,\omega_i+1} \in \mathbb F_{q^{t_l}}$ are linearly independent over $\mathbb F_q$, and $g_{i,j}\in \mathbb F_{q^{t_{x+1}}}$ ($i+1\leq j\leq \omega_i+1$) if $t_x< j+1\leq t_{x+1}$ for some $0\leq x< l$;
\item[$(2)$] the sub-matrix of $\textbf{G}_i$ obtained by removing its first $i$ rows and the leftmost $i$ columns produces a $\mathcal G[t_l \times (\eta-r),d+i]_q$ code.
\end{itemize}
When $i=0$, by \eqref{basis-1} and \eqref{basis-2}, $\textbf{G}_i$ is just $\textbf{G}_0$. By using exactly the same argument as that in the proof of Lemma $3.11$ in \cite{lf}, via elementary row transformations together with adding one new column, one can see how to obtain  $\textbf{G}_{i+1}$ from  $\textbf{G}_i$ for $0\leq i\leq r-1$. Here the reader only need to examine that if elements in the $j$-column of $\textbf{G}_i$ belong to $\mathbb F_{q^{t_{x+1}}}$, then after elementary row transformations used in the proof of Lemma $3.11$ in \cite{lf}, they are still in $\mathbb F_{q^{t_{x+1}}}$. And since $t_{l-1}<\eta-r\leq t_l$, elements in the added new column are taken all from $\mathbb F_{q^{t_{l}}}$.

Finally, we can choose an invertible matrix $\textbf{T} \in \mathbb F_{q^{t_1}}^{(\kappa-r) \times (\kappa-r)}$ such that
\begin{center}
 $\textbf{G}=\left(
                                    \begin{array}{cc}
                                     \textbf{I}_{r \times r} &  \\
                                      & \textbf{T} \\
                                    \end{array}
                                  \right) \cdot \textbf{G}_r$
\end{center}
is our required matrix. \qed
\end{proof}

Let $\alpha_1, \alpha_2, \ldots, \alpha_{n} \in \mathbb F_{q^m}$. Write ${\rm span}_{\mathbb F_q}(\alpha_1, \alpha_2, \ldots, \alpha_{n})\triangleq \{w_1 \alpha_1+ w_2\alpha_2+ \cdots+ w_{n}\alpha_{n}: w_i\in \mathbb F_q, 1\leq i\leq n\}$. The idea of the following theorem comes from Theorems $3.2$ and $3.6$ in \cite{zg}.

\begin{theorem}\label{the:from sys MRD-new2}
Let $l$ be a positive integer and $1=t_0<t_1<t_2<\cdots<t_l$ be integers such that $t_1\mid t_2\mid \cdots\mid t_l$. When $l>1$, let $t_{2}=s_{2}t_{1}$. Let $r$ be a nonnegative integer and $\delta$, $n$, $k$ be positive integers satisfying $r+1\leq \delta\leq n-r$, $t_{l-1}<n-r\leq t_l$, $k=n-\delta+1$ and $k\leq t_1$. Let $\mathcal F=[\gamma_0,\gamma_1,\ldots,\gamma_{n-1}]$ be an $m\times n$ Ferrers diagram $(m=\gamma_{n-1})$ satisfying
\begin{itemize}
\item[$(1)$] $\gamma_{k-1}\leq wt_1$,%when $l\geq 2$ or $r\leq 1$, $\gamma_{k-1}\leq wt_1$ when $l=1$ and $r\geq 2$,
\item[$(2)$] $\gamma_{k}\geq wt_1$ for $k<t_1$ and $\delta\geq 2$,
\item[$(3)$] $\gamma_{t_{\theta}}\geq t_{\theta+1}$ for $1\leq \theta \leq l-1$,
\item[$(4)$] $\gamma_{n-r+h}\geq t_l+\sum_{j=0}^{h} \gamma_j$ for $0\leq h\leq r-1$,
\end{itemize}
for some $w\in\{1,2,\ldots,s_2\}$ and for $w=1$ if $l=1$. Then there exists an optimal $[\mathcal F, \sum_{i=0}^{k-1} \gamma_i, \delta]_q$ code for any prime power $q$.
\end{theorem}

\begin{proof}
%By Theorem \ref{thm:shortening}, when $\delta\in\{1,2\}$, there always exists an optimal $[\mathcal F, \sum_{i=0}^{n-\delta} \gamma_i, \delta]_q$ code for any prime power $q$. In what follows, assume that $\delta\geq 3$.
When $\delta=1$, the conclusion is trivial. Assume that $\delta\geq 2$. Let $(\beta_1=1,\beta_2,\ldots,\beta_{t_l})$ be an ordered basis of $\mathbb F_{q^{t_l}}$ over $\mathbb F_q$ with respect to $(t_0,t_1,\ldots,t_l)$. Note that $\delta\leq n-r$, so $r\leq n-\delta<n-\delta+1=k$. Applying Lemma \ref{lem:subcode from Gab} by taking $\eta=n$, $d=\delta-r$ and $\kappa=k$, we can obtain a matrix $\boldsymbol G=(\boldsymbol I_{k}|\boldsymbol A_1|\cdots|\boldsymbol A_{l})\in \mathbb F_{q^{t_l}}^{k \times n}$, where $$\boldsymbol A_1\in \left\{
\begin{array}{lll}
\mathbb F^{k\times (t_1-k)}_{q^{t_1}}, &  {\rm if}\ l\geq 2, \\
\mathbb F^{k\times (n-k)}_{q^{t_1}}, & {\rm if}\ l=1,
\end{array}
\right.
$$
$\boldsymbol A_x\in \mathbb F^{k\times (t_x-t_{x-1})}_{q^{t_x}}$ for $2\leq x\leq l-1$, and $\boldsymbol A_l\in \mathbb F^{k\times (n-t_{l-1})}_{q^{t_l}}$ if $l\geq 2$, satisfying that for each $0\leq \nu\leq r$, the sub-matrix obtained by removing the first $\nu$ rows, the leftmost $\nu$ columns and the rightmost $r-\nu$ columns of $\boldsymbol{G}$ produces a systematic MRD$[t_l \times (n-r),\delta-r+\nu]_q$ code.

For $i\in [k]$, fix any $\gamma_i\leq wt_1$, where $w=1$ if $l=1$, and $w\in\{1,2,\ldots,s_2\}$ if $l\geq 2$. Let $${U}=\left\{(u_{0},u_{1},\ldots,u_{k-1}) \in \mathbb F_{q^{t_l}}^k:
u_i\in {\rm span}_{{\mathbb F}_q}(\beta_1,\beta_2,\ldots,\beta_{\gamma_i}), i\in [k] \right\}.$$
Then for $i\in [k]$, $\Psi_{t_l}(u_i)=(u_{i,0},u_{i,1},\ldots,u_{i,\gamma_{i}-1},0,\ldots,0)^T$ for some $u_{i,j}\in \mathbb F_q$ where $0\leq j\leq \gamma_i-1$. Let $\overline{\Psi}_{t_l}(u_h)=(u_{h,0},u_{h,1},\ldots,u_{h,\gamma_{h}-1})^T$ for $0\leq h\leq r-1$ (note that $r<k$). Let $m'=t_l+\sum_{h=0}^{r-1} \gamma_h$. Let
\begin{center}
${\cal C}=\left\{
\left(
\begin{array}{c}
\displaystyle\frac{
\begin{array}{c}
\Psi_{t_l}( \boldsymbol{uG})
\end{array}
}{\frac{
\begin{array}{ccccccc}
0 & \cdots & 0 & \overline{\Psi}_{t_l}(u_0)  & \overline{\Psi}_{t_l}(u_1)  &  \cdots & \overline{\Psi}_{t_l}(u_{r-1}) \\
0 & \cdots  & 0 & 0 & \overline{\Psi}_{t_l}(u_0)  & \ddots & \overline{\Psi}_{t_l}(u_{r-2}) \\
\vdots &   & \vdots & \vdots &  \vdots & \ddots & \vdots \\
0 & \cdots & 0 & 0 &  0 &  \cdots & \overline{\Psi}_{t_l}(u_0) \\
\end{array}}
{\begin{array}{c}
\boldsymbol{O}_{(m-m')\times n}
\end{array}}
} \\
    \end{array}
  \right)\in \mathbb F^{m\times n}_q: \boldsymbol{u} \in U
\right\}$,
\end{center}
where $\boldsymbol{O}_{(m-m')\times n}$ is a zero matrix. Note that the four conditions in the assumption imply $m=\gamma_{n-1}\geq m'$: when $r>0$, by Condition $(4)$, $\gamma_{n-1}\geq m'$; when $r=0$ and $l=1$, by Condition $(2)$, $\gamma_{n-1}\geq m'=t_1$; when $r=0$ and $l>1$, by Condition $(3)$, $\gamma_{n-1}\geq m'=t_l$.

Let $\gamma'_i$, $i\in [n]$, be the maximal valid length of the $(i+1)$st column of the matrices in $\cal C$. We shall show that ${\cal C}$ is an optimal $[\mathcal F', \sum_{i=0}^{k-1} \gamma_i, \delta]_q$ code, where $\mathcal F'=[\gamma'_0,\gamma'_1,\ldots,\gamma'_{n-1}]$ satisfies $\gamma'_i=\gamma_i$ for each $0\leq i\leq k-1$ and $\gamma'_i\leq \gamma_i$ for each $k\leq i\leq n-1$. That yields $\mathcal F'\subseteq \mathcal F$ and implies the existence of an optimal $[\mathcal F, \sum_{i=0}^{k-1} \gamma_i, \delta]_q$ code by examining the value of $v_0$ in Lemma \ref{lem:upper bound}.

First, we analyze the number of dots in each column of $\mathcal F'$. Take any $\boldsymbol{u}=(u_0,u_1,\ldots$, $u_{k-1})\in U$ and set $\boldsymbol{uG}=(e_0,e_1,\ldots,e_{n-1})$.

For $i\in [k]$, we have $e_i=u_i$. So $\Psi_{t_l}(e_i)=\Psi_{t_l}(u_i)=(u_{i,0},u_{i,1},\ldots,u_{i,\gamma_i-1},0,\ldots,0)$. Since $k=n-\delta+1$ and $\delta\geq r+1$, we have
\begin{eqnarray}\label{eq}
n-r\geq k.
\end{eqnarray}
It follows that %the number of dots in the $i$-th column of $\mathcal F'$ can be taken as $\gamma_i$,
$\gamma'_i$ can be taken as $\gamma_i$. This ensures the optimality of $\cal C$ by Lemma \ref{lem:upper bound}.

Let the $k$-th column of $\boldsymbol G$ be $(b_0,b_1,\ldots,b_{k-1})^T$. Then $e_{k}=\sum_{i=0}^{k-1} u_i b_{i}$, and so $\Psi_{t_l}(e_{k})=\sum_{i=0}^{k-1} \Psi_{t_l}(u_ib_{i})$. For $i\in [k]$, $\Psi_{t_l}(u_i)=(u_{i,0},u_{i,1},\ldots,$ $u_{i,\gamma_{i}-1},0,$ $\ldots,0)^T$ implies  $u_i=u_{i,0}\beta_1+u_{i,1}\beta_2+\cdots+u_{i,\gamma_i-1}\beta_{\gamma_i}$. By \eqref{eq}, $n-r\geq k$, so we distinguish two cases. \underline{First case: $n-r\geq k+1$}. \underline{If $k<t_1$}, then the $k$-th column of $\boldsymbol G$ comes from $\boldsymbol A_1$, and so $b_i\in \mathbb F_{q^{t_1}}$ for $i\in [k]$. For $i\in [k]$, we have
$$u_ib_i=u_{i,0}\beta_1b_i+u_{i,1}\beta_2b_i+\cdots+u_{i,\gamma_i-1}\beta_{\gamma_i}b_i,$$
where $b_i\in \mathbb F_{q^{t_1}}$ and $u_{i,j}\in \mathbb F_{q}$ for $0\leq j\leq \gamma_i-1$. By Condition $(1)$, $\gamma_i\leq wt_1$ for each $i\in [k]$, so by Proposition \ref{prop:basis}(1), we have
$$u_{i,j}\beta_{j+1}b_i\in {\rm span}_{\mathbb F_q}(\beta_1, \beta_2, \ldots, \beta_{wt_1})$$
for each $0\leq j\leq \gamma_i-1$. Thus $\Psi_{t_l}(e_{k})$ has a valid length of at most $wt_1$. By Condition $(2)$, $\gamma'_k\leq \gamma_k$.
\underline{If $k=t_1$ and $l=1$}, then $n-r\geq k+1=t_1+1$. Since $l=1$ implies $n-r\leq t_1$, a contradiction occurs.
\underline{If $k=t_1$ and $l\geq 2$}, then $t_1<n$. The $k$-th column of $\boldsymbol G$ comes from $\boldsymbol A_2$, and so $b_i\in \mathbb F_{q^{t_2}}$ for $i\in [k]$. Thus by Proposition \ref{prop:basis}(2), $\Psi_{t_l}(e_{k})$ has a valid length of at most $t_2$. By Condition $(3)$, $\gamma'_k=\gamma'_{t_1}\leq \gamma_{t_1}=\gamma_k$. \underline{Second case: $n-r=k$}. Since $t_{l-1}<n-r=k\leq t_l$ and $k\leq t_1$, we have $l=1$. It follows that $\Psi_{t_l}(e_{k})$ has a valid length of at most $t_1$. Since $\overline{\Psi}_{t_l}(u_{0})$ has a valid length of at most $\gamma_0$, we obtain $\gamma'_{k}=\gamma'_{n-r}= t_1+\gamma_0$. By Condition $(4)$, $\gamma'_k=\gamma'_{n-r}\leq \gamma_k$.

For $l\geq 2$ and $1\leq \theta\leq l-1$, since $t_{l-1}<n-r\leq t_l$, we have $n-r>t_{\theta}$. Let the $t_{\theta}$-th column of $\boldsymbol G$ be $(b_{t_{\theta},0}$, $b_{t_{\theta},1},\ldots$, $b_{t_{\theta},k-1})^T$, which is the first column of $\boldsymbol A_{\theta+1}$. Then by Proposition \ref{prop:basis}(2), $e_{t_{\theta}}=\sum_{i=0}^{k-1} u_i b_{t_{\theta},i}\in \mathbb F_{q^{t_{\theta+1}}}$. Thus $\Psi_{t_l}(e_{t_{\theta}})$ has a valid length of at most $t_{\theta+1}$. By Condition $(3)$, $\gamma'_{t_{\theta}}\leq \gamma_{t_{\theta}}$.

For $0\leq h\leq r-1$, $\Psi_{t_l}(e_{n-r+h})$ has a valid length of at most $t_l$ and  $\overline{\Psi}_{t_l}(u_{h})$ has a valid length of at most $\gamma_h$. Thus we can take $\gamma'_{n-r+h}= t_l+\sum_{j=0}^{h} \gamma_j$ for $0\leq h\leq r-1$. By Condition $(4)$, $\gamma'_{n-r+h}\leq \gamma_{n-r+h}$.

Next, one can easily verify the linearity and the dimension of the code $\mathcal C$. Finally it suffices to examine the minimum rank weight of any nonzero codeword $\boldsymbol{C}$ from $\mathcal C$.

Let $\boldsymbol{C}$ be formed by $\boldsymbol{uG}=(u_0, u_1, \ldots, u_{k-1}) \boldsymbol{G}$. Let $i^*=\min\{i\in [k],u_i\neq 0, u_j=0 {~\rm for~any~} j<i\}$. Then $\boldsymbol{uG}=(0, \ldots,0,u_{i^*}, \ldots, u_{k-1}) \boldsymbol{G}$.

If $i^*< r$, then let $\Psi^*_{t_l}(\boldsymbol{uG})$ be an $t_l\times (n-r)$ matrix obtained by removing the leftmost $i^*$ columns and the rightmost $r-i^*$ columns of $\Psi_{t_l}(\boldsymbol{uG})$. By Lemma \ref{lem:subcode from Gab}, $\Psi_{t_l}^*(\boldsymbol{uG})$ is a codeword of an MRD$[t_l \times (n-r),\delta-r+i^*]_q$ code, whose generator matrix can be obtained by removing the first $i^*$ rows, the leftmost $i^*$ columns and the rightmost $r-i^*$ columns of $\boldsymbol{G}$. Thus rank$(\Psi_{t_l}^*(\boldsymbol{uG}))\geq \delta-r+i^*$. Furthermore, under the broken line of $\boldsymbol{C}$, since $\overline{\Psi}_{t_l}(u_{i^*})$ is a nonzero vector, the rightmost $r-i^*$ columns can contribute rank $r-i^*$. Therefore, rank$(\boldsymbol{C})\geq {\rm rank}(\Psi_{t_l}(\boldsymbol{uG})^*)+r-i^* \geq \delta-r+i^*+r-i^*=\delta$.

If $i^*\geq r$, then let $\Psi^*_{t_l}(\boldsymbol{uG})$ be the $t_l\times (n-r)$ matrix obtained by removing the leftmost $r$ columns of $\Psi_{t_l}(\boldsymbol{uG})$. By Lemma \ref{lem:subcode from Gab}, $\Psi_{t_l}^*(\boldsymbol{uG})$ is a codeword of an MRD$[t_l \times (n-r),\delta]_q$ code, whose generator matrix can be obtained by removing the first $r$ rows and the leftmost $r$ columns of $\boldsymbol{G}$. Thus rank$(\boldsymbol{C})\geq\delta$. \qed
\end{proof}

\begin{remark}\label{rek:main}
\begin{itemize}
\item[$(1)$] Take $l=1$, $r=0$ and $t_1=n\leq m$ in Theorem $\ref{the:from sys MRD-new2}$ to obtain Theorem $\ref{thm:shortening}$, i.e., Theorem $3$ in {\rm \cite{egrw}}.
\item[$(2)$] Take $l=1$ and $t_1=n-r$ in Theorem $\ref{the:from sys MRD-new2}$ to obtain Theorem $3.13$ in {\rm \cite{lf}}, which is a generalization of Theorem $\ref{thm:subcodes from Gab}$.
\item[$(3)$] Take $w=1$ and $r=0$ in Theorem $\ref{the:from sys MRD-new2}$ to obtain Theorem $3.2$ in {\rm \cite{zg}}, which requires each of the first $k$ columns of $\cal F$ contains at most $t_1$ dots. When $l\geq 2$, Theorem $\ref{the:from sys MRD-new2}$ relaxes this restriction, that is to say, each of the first $k$ columns of $\cal F$ contains at most $t_2$ dots, where $t_1\mid t_2$ and $t_1<t_2$.
\item[$(4)$] Take $w=1$ and $r=1$ in Theorem $\ref{the:from sys MRD-new2}$ to obtain Theorem $3.6$ in {\rm \cite{zg}}.
\end{itemize}
\end{remark}

\begin{corollary}\label{the:from sys MRD-3}
Let $r$ be a nonnegative integer and $m$, $n$, $\delta$, $k$, $t_1$, $t_2$ be positive integers satisfying $r+1\leq \delta \leq n-r$, $k=n-\delta+1$, $k\leq t_1<n-r\leq t_2\leq m$ and $t_2=st_1$. If an $m\times n$ Ferrers diagram $\mathcal F=[\gamma_0,\gamma_1,\ldots,\gamma_{n-1}]$ satisfies
\begin{itemize}
\item[$(1)$] $\gamma_{k-1}\leq wt_1$, %for $1\leq i\leq k$,
\item[$(2)$] $\gamma_{k}\geq wt_1$ when $k<t_1$,%for $k+1\leq i\leq t$,
\item[$(3)$] $\gamma_{t_1}\geq t_2$, %for $t+1\leq i\leq n$,
\item[$(4)$] $\gamma_{n-r+h}\geq t_2+\sum_{j=0}^{h} \gamma_j$ for $0\leq h\leq r-1$,
\end{itemize} %where $l\in \{1,\ldots,s\}$,
for some $w\in\{1,2,\ldots,s\}$, then there exists an optimal $[\mathcal F, \sum_{i=0}^{k-1} \gamma_i, \delta]_q$ code for any prime power $q$.
\end{corollary}

\proof Apply Theorem \ref{the:from sys MRD-new2} with $l=2$. \qed

\begin{example}\label{eg:sys MDS-5}
Let $\mathcal F=[10,10,10,10,10,15,15,15,15,15,15,15,15,15,15]$ be a $15\times 15$ Ferrers diagram.
Then apply Corollary $\ref{the:from sys MRD-3}$ with $r=0$, $s=3$, $t_1=5$, $t_2=15$ and $w=2$ to obtain an optimal $[\mathcal F, 40, 12]_q$ code for any prime power $q$.
\end{example}

\begin{example}\label{ex-sub}
Let $3t\geq n$ and $$\mathcal F=[\underbrace{\gamma_0,\ldots,\gamma_{k-1}}_{k}, \underbrace{2t,\ldots,2t}_{t-k}, \underbrace{3t,\ldots,3t}_{n-t}]$$ be a $3t\times n$ Ferrers diagram, where $k=n-\delta+1$, $1\leq \delta\leq n$ and $k\leq t< n$. Then apply Corollary $\ref{the:from sys MRD-3}$ with $r=0$, $s=3$, $t_1=t$, $t_2=3t$ and $w=2$ to obtain an optimal $[\mathcal F, \sum_{i=0}^{k-1} \gamma_i, \delta]_q$ code for any prime power $q$.
\end{example}

\begin{example}\label{eg:sys MDS-4}
For any even integer $n\geq 10$, let
\[\mathcal F=
\begin{array}{c@{\hspace{-5pt}}c@{\hspace{-5pt}}c}
&\begin{array}{cc}
\overbrace{\rule{25mm}{0mm}}^{\frac{n}{2}-1}&
\overbrace{\rule{20mm}{0mm}}^{\frac{n}{2}+1}
\end{array}
\\
\begin{array}{l}
%m_1\left.\rule{0mm}{8mm}\right\{\\
\end{array}
&
\begin{array}{ccccccccc}
 \bullet & \bullet & \bullet & \cdots & \bullet & \bullet & \cdots & \bullet& \bullet \\
 &  \bullet & \bullet & \cdots & \bullet & \bullet & \cdots & \bullet& \bullet \\
 & &  \vdots &  & \vdots & \vdots &  & \vdots & \vdots \\
 & &  \bullet & \cdots & \bullet & \bullet & \cdots & \bullet& \bullet \\
 & &   &  &  & \bullet & \cdots & \bullet& \bullet \\
 & &   &  &  & \vdots &  & \vdots & \vdots \\
 & &   &  &  & \bullet & \cdots & \bullet& \bullet \\
 &  &   &  &  &  & & \bullet  & \bullet \\
 &  &   &  &  &  &  & & \bullet  \\
  &  &   &  &  &  &  & & \bullet
\end{array}
& \begin{array}{l}
\left.\rule{0mm}{9.92mm}\right\}\frac{n}{2}-1\\
\\\left.\rule{0mm}{15.15mm}\right\}\frac{n}{2}+2
\end{array}
%\\[-5pt]
%\begin{array}{c}
%\hspace{3.75cm}
%\underbrace{\rule{15mm}{0mm}}_{n_3}\end{array} &
\end{array}
\]
be an $(n+1)\times n$ Ferrers diagram. Then apply Corollary $\ref{the:from sys MRD-3}$ with $r=2$, $s=2$, $t_1=\frac{n}{2}-1$, $t_2=n-2$ and $w=1$ to obtain an optimal $[\mathcal F, \frac{n(n-10)}{4}+7, \frac{n}{2}+3]_q$ code for any prime power $q$.
\end{example}

%\begin{lemma}[Theorem 5 in \cite{Sheekey}]\label{sheekey16}
%Let $m\geq n\geq \delta$. Set $k=n-\delta+1$ and
%\begin{center}
%$\boldsymbol{G}=\left(
%                  \begin{array}{cccc}
%                    g_{0} & g_{1} & \cdots & g_{n-1} \\
%                    g_{0}^{[1]} & g_{1}^{[1]} & \cdots & g_{n-1}^{[1]} \\
%                    \vdots & \vdots & \cdots & \vdots \\
%                    g_{0}^{[k]} & g_{1}^{[k]} & \cdots & g_{n-1}^{[k]} \\
%                  \end{array}
%                \right)$,
%\end{center} where $g_{0}, g_{1}, \cdots, g_{n-1}$ is an ordered basis of $\mathbb F_{q^{t_l}}$ over $\mathbb F_q$ {\em with respect to $(t_0,t_1,\ldots,t_l)$}.
%Set
%\begin{center}
%$\mathcal C=\{\Psi_m(\boldsymbol{uG}):\boldsymbol u=(u_0,u_1,\ldots,u_{k-1},\eta u_{0}),u_i\in \mathbb F_{q^m}, i\in [k]\}$,
%\end{center}where $\eta^{\frac{q^m-1}{q-1}}\neq (-1)^{km}$, then $\mathcal C$ is an MRD$[m\times n,\delta]_q$ code.
%\end{lemma}

%\begin{remark}
%For the Ferrers diagram in Corollary , if set $c_{k}=c_{k+1}$, there exists an optimal FDRM code. However, it is almost impossible to construct optimal FDRM codes by all known constructions based on subcodes of MRD codes because of Lemma $3.5$ in \cite{al}.
%\end{remark}

\section{Constructions via different representations of elements of a finite field}

In this section, based on two different ways to represent elements of a finite field $\mathbb F_{q^m}$ (vector representation and matrix representation), we give two constructions for FDRM codes, where the first one is also given by Zhang and Ge recently (see Theorem 3.9 in \cite{zg}).

\begin{theorem}[Based on vector representation]\label{lift-product-1}
If there exists an $[\mathcal F,k,\delta]_{q^m}$ code, where $\mathcal F=[\gamma_0,\gamma_1,\ldots,\gamma_{n-1}]$, then there exists an $[\mathcal F',mk,\delta]_q$ code, where $\mathcal F'=[m\gamma_0,m\gamma_1,\linebreak \ldots,m\gamma_{n-1}]$.
\end{theorem}

\begin{proof}
By (\ref{iso}), each element in $\mathbb F_{q^m}$ can be represented as a column vector in $\mathbb F_{q}^{m \times 1}$ via the bijection $\Psi_m$, and $\Psi_m$ satisfies linearity. Let $\boldsymbol C$ be a codeword of the given $[\mathcal F,k,\delta]_{q^m}$ code $\mathcal C$, where $\mathcal F$ is a $\gamma_{n-1}\times n$ Ferrers diagram. Let
\begin{center}
$\boldsymbol D_C=\left(
   \begin{array}{cccc}
     \Psi_m(\boldsymbol C(0,0)) & \Psi_m(\boldsymbol C(0,1)) & \cdots & \Psi_m(\boldsymbol C(0,n-1)) \\
     \Psi_m(\boldsymbol C(1,0)) & \Psi_m(\boldsymbol C(1,1)) & \cdots & \Psi_m(\boldsymbol C(1,n-1)) \\
    \vdots & \vdots & \ddots  & \vdots \\
    \Psi_m(\boldsymbol C(\gamma_{n-1}-1,0)) & \Psi_m(\boldsymbol C(\gamma_{n-1}-1,1)) & \cdots & \Psi_m(\boldsymbol C(\gamma_{n-1}-1,n-1)) \\
   \end{array}
 \right)$,
\end{center}
and $\mathcal C'=\{\boldsymbol D_{\boldsymbol C}:\boldsymbol C\in \mathcal C\}$. It is readily checked that $\mathcal C'$ is an $[\mathcal F',mk,\delta]_q$ code, where $\mathcal F'=[m\gamma_0,m\gamma_1,\ldots,m\gamma_{n-1}]$. \qed
\end{proof}

Apart from (\ref{iso}), a possibility of representing the elements of $\mathbb F_{q^m}$ is given by means of matrices (see Chapter $2.5$ in \cite{ln}). The field $\mathbb F_{q^m}$ is isomorphic to a suitable subset of $\mathbb F^{m\times m}_q$. We can give this well-known fact as follows. Let
$g(x)=x^m+g_{m-1}x^{m-1}+\cdots+g_1x+g_0 \in \mathbb F_q[x]$ be a primitive polynomial over $\mathbb F_{q}$, whose {\em companion matrix} is
\begin{center}
$\boldsymbol G=\left(
   \begin{array}{cccccc}
     0 & 0 & 0 & \cdots & 0 & -g_0 \\
     1 & 0 & 0 & \cdots & 0 & -g_1 \\
     0 & 1 & 0 & \cdots & 0 & -g_2 \\
     0 & 0 & 1 & \cdots & 0 & -g_3 \\
    \vdots & \vdots & \vdots & \ddots & \vdots & \vdots \\
    0 & 0 & 0 & \cdots & 1 & -g_{m-1} \\
   \end{array}
 \right)$.
\end{center}
By the Cayley-Hamilton theorem in linear algebra, $\boldsymbol G$ is a root of $g(x)$. The set $\mathcal A=\{\boldsymbol G^i:0\leq i\leq q^m-2\}\cup\{\boldsymbol 0\}$ equipped with matrix addition and matrix multiplication is isomorphic to $\mathbb F_{q^m}$. Let $\alpha$ be a primitive element of $\mathbb F^*_{q^m}$. Then $\mathbb F_{q^m}=\{1,\alpha,\alpha^2,\ldots,\alpha^{q^m-2}\}\cup\{0\}$. Let $\Pi_m$ be the field isomorphism from $\mathbb F_{q^m}$ to $\mathcal A$ satisfying $\Pi_m(0)=\boldsymbol O$ and $\Pi_m(\alpha^i)= \boldsymbol G^i$ for $0\leq i\leq q^m-2$. Clearly, $\Pi_m(c\alpha^i)=c\Pi_m(\alpha^i)$ for $c\in \mathbb F_p$, where $p$ is a prime dividing $q$.

%\begin{center}
%$\pi_m: \mathbb F_{q^m}\longrightarrow \mathcal A$
%\end{center}
%\begin{center}
%$a_i\longrightarrow \boldsymbol G^i$,
%\end{center}
%where $i\in \{0,1,\ldots,q^r-2\}$, $\pi_r(0)=\boldsymbol 0$.
%$\mathcal A$ is called matrix representation of $\mathbb F_{q^r}$ over $\mathbb F_q$. Therefore, it seems that there is a natural and trivial lifting an FDRM code $\mathcal C'\subseteq \mathbb F^{m\times n}_{q^r}$ with the minimum rank distance $\delta$ over $\mathbb F_{q^r}$ to another FDRM code $\mathcal C\subseteq \mathbb F^{mr\times nr}_q$ with the minimum rank distance $r\delta $ over $\mathbb F_q$.

\begin{lemma}\label{lem-new}
Let $p$ be a prime and $q=p^l$. Let $\boldsymbol A \in \mathbb F^{s\times t}_{q^m}$ and
 \begin{center}
$\Pi_m(\boldsymbol A)\triangleq\left(
   \begin{array}{cccc}
     \Pi_m(\boldsymbol A(0,0)) & \Pi_m(\boldsymbol A(0,1)) & \cdots & \Pi_m(\boldsymbol A(0,t-1)) \\
     \Pi_m(\boldsymbol A(1,0)) & \Pi_m(\boldsymbol A(1,1)) & \cdots & \Pi_m(\boldsymbol A(1,t-1)) \\
    \vdots & \vdots & \ddots  & \vdots \\
    \Pi_m(\boldsymbol A(s-1,0)) & \Pi_m(\boldsymbol A(s-1,1)) & \cdots & \Pi_m(\boldsymbol A(s-1,t-1)) \\
   \end{array}
 \right)$.
\end{center}
If ${\rm rank}(\boldsymbol A)\geq \delta$ in $\mathbb F_{q^m}$, then ${\rm rank}(\Pi_m(\boldsymbol A))\geq m\delta$ in $\mathbb F_{q}$.
\end{lemma}

\begin{proof}
Choose an invertible $\delta\times \delta$ submatrix $\boldsymbol D$ of $\boldsymbol A$. Then $\Pi_m(\boldsymbol D)$ consists of mutually commuting blocks and therefore Theorem 1 in \cite{sil} applies. This gives us $\det(\Pi_m(\boldsymbol D))=\det(\Pi_m(\det(\boldsymbol D))$. Since $\det(\boldsymbol D)$ is nonzero, so is $\Pi_m(\det(\boldsymbol D))$ and hence it is even invertible since $\Pi_m$ is a field automorphism. Thus $\det(\Pi_m(\boldsymbol D))$ is nonzero. \qed
\end{proof}

\begin{theorem}[Based on matrix representation]\label{lift-product}
If there exists an $[\mathcal F,k,\delta]_{q^m}$ code, where $\mathcal F=[\gamma_0,\gamma_1,\ldots,\gamma_{n-1}]$, then there exists an $[\mathcal F',mk,m\delta]_q$ code, where $$\mathcal F'=[\underbrace{m\gamma_0,\ldots,m\gamma_0}_{m}, \underbrace{m\gamma_1,\ldots,m\gamma_1}_{m}, \ldots,\underbrace{m\gamma_{n-1},\ldots,m\gamma_{n-1}}_{m}].$$
\end{theorem}

\begin{proof}
Let $\boldsymbol C$ be a codeword of the given $[\mathcal F,k,\delta]_{q^m}$ code $\mathcal C$, where $\mathcal F$ is a $\gamma_{n-1}\times n$ Ferrers diagram. Let
\begin{center}
$\boldsymbol D_C=\left(
   \begin{array}{cccc}
     \Pi_m(\boldsymbol C(0,0)) & \Pi_m(\boldsymbol C(0,1)) & \cdots & \Pi_m(\boldsymbol C(0,n-1)) \\
     \Pi_m(\boldsymbol C(1,0)) & \Pi_m(\boldsymbol C(1,1)) & \cdots & \Pi_m(\boldsymbol C(1,n-1)) \\
    \vdots & \vdots & \ddots  & \vdots \\
    \Pi_m(\boldsymbol C(\gamma_{n-1}-1,0)) & \Pi_m(\boldsymbol C(\gamma_{n-1}-1,1)) & \cdots & \Pi_m(\boldsymbol C(\gamma_{n-1}-1,n-1)) \\
   \end{array}
 \right)$,
\end{center}
and $\mathcal C'=\{\boldsymbol D_{\boldsymbol C}:\boldsymbol C\in \mathcal C\}$. By Lemma \ref{lem-new}, it is readily checked that $\mathcal C'$ is an $[\mathcal F',mk,m\delta]_q$ code. \qed
\end{proof}

\begin{remark}
The idea of Theorem $\ref{lift-product}$ is from Proposition $3.1$ in {\rm \cite{oo}}. Also, Theorems $\ref{lift-product-1}$ and $\ref{lift-product}$ can be generalized to deal with non-linear FDRM codes.
\end{remark}

\begin{theorem}\label{thm:optimal matrix rep}
If there exists an optimal $[\mathcal F,\gamma_0,n]_{q^m}$ code, where $\mathcal F=[\gamma_0,\gamma_1,\ldots,\gamma_{n-1}]$, then there exists an optimal $[\mathcal F',m\gamma_0,mn]_q$ code, where $$\mathcal F'=[\underbrace{m\gamma_0,\ldots,m\gamma_0}_{m}, \underbrace{m\gamma_1,\ldots,m\gamma_1}_{m}, \ldots,\underbrace{m\gamma_{n-1},\ldots,m\gamma_{n-1}}_{m}].$$
\end{theorem}

\proof Start from the given optimal $[\mathcal F,\gamma_0,n]_{q^m}$ code, whose dimension can be obtained by deleting its rightmost $n-1$ columns. Then apply Theorem \ref{lift-product} to obtain an $[\mathcal F',m\gamma_0,mn]_q$ code, whose optimality can be obtained by deleting its rightmost $mn-1$ columns.  \qed

\begin{example}\label{eg:lift-1}
By Example $\ref{eg:2,3,3,5}$, there exists an optimal $[\mathcal F,2,4]_{q^2}$ code for any prime power $q$, where $\mathcal F=[2,3,3,5]$ is a Ferrers diagram. Then apply Theorem $\ref{thm:optimal matrix rep}$ with $m=2$ to obtain an optimal $[\mathcal F',4,8]_{q}$ code, where $\mathcal F'=[4,4,6,6,6,6,10,10]$.
\end{example}

\section{Conclusion}

%Five constructions for FDRM codes are presented in this paper. All known constructions from \cite{al,egrw,es09,lf,gr,st15,zg} cannot produce optimal FDRM codes obtained from Examples \ref{fdrm5}, \ref{ex-sub}, \ref{eg:sys MDS-4}, \ref{eg:sys MDS-5}, \ref{eg:lift-1} and \ref{eg:pro-1}.
Four constructions for FDRM codes are presented in this paper. The first one makes use of a characterization on generator matrices of a class of systematic MRD codes. By introducing restricted Gabidulin codes, the second construction is presented, which unifies many known constructions for FDRM codes. The third and fourth constructions are based on two different ways to represent elements of a finite field $\mathbb F_{q^m}$ (vector representation and matrix representation).
%The last one is based on Ferrers diagram Kronecker products.

Theorem \ref{cfdrm} was established by using a description on generator matrices of a class of systematic MRD codes shown in Lemma \ref{lem:subcode from MRD}. Giving more characterization on generator matrices of systematic MRD codes would be helpful to obtain more optimal FDRM codes.

Finally we summarize all the main constructions for FDRM codes as follows: $(1)$ constructions based on subcodes of MRD codes (see Theorem 3.6 in \cite{al}, Construction 3.5 in \cite{lf}, Theorems \ref{cfdrm} and \ref{the:from sys MRD-new2}); $(2)$ the construction based on MDS codes (see Construction 1 in \cite{egrw}); $(3)$ constructions by combining FDRM codes (see Constructions 4.7, 4.10, and 4.13 in \cite{lf}); $(4)$ constructions via different representations of elements of a finite field (see Theorems \ref{lift-product-1} and \ref{lift-product}).

\subsection*{Acknowledgements}
The authors express their gratitude to the anonymous referees for their detailed and constructive comments which are very helpful to the improvement of the paper. Especially, thank one of the referees for pointing out the reference \cite{sil} and simplifying the proof of Lemma \ref{lem-new}.

\end{document}